\documentclass[12pt,a4paper]{amsart}
\newcommand{\Title}{A Complete Characterization of Cartan Inclusions of Finite Dimensional $C^*$-algebras}%
\newcommand{\ShortTitle}{\Title}%

\newcommand{\AuthorOne}{Indrajit Ghosh}%
\newcommand{\AuthorOneAddr}{%
	Department of Mathematics and Statistics, Indian Institute of Technology Kanpur, Uttar Pradesh 208 016, India
}%

\newcommand{\AuthorOneEmail}{%
	indrajitghosh912@gmail.com, indrajitg@iitk.ac.in
}%

\newcommand{\AuthorTwo}{Sumit Kumar}%
\newcommand{\AuthorTwoAddr}{Department of Mathematics and Statistics, Indian Institute of Technology Kanpur, Uttar Pradesh 208 016, India
}%
\newcommand{\AuthorTwoEmail}{sumitkumar.sk809@gmail.com}

\newcommand{\SubjectClassText}{Primary 46L05, 47L40; Secondary 46L10}

\newcommand{\Keywords}{Generalised Cartan Subalgebras, Regularity, $C^*$-algebra}

\newcommand{\pdfTitle}{\Title}
\newcommand{\pdfAuthor}{Indrajit Ghosh}
\newcommand{\pdfSubject}{Mathematics, Research paper}
\newcommand{\pdfKeywords}{\Keywords}
\newcommand{\pdfCreator}{TeXlive}
\newcommand{\pdfCreationDate}{\today}
\newcommand{\pdfColorLink}{true}
\newcommand{\pdfLinkColor}{cyan}
\newcommand{\pdfUrlColor}{blue}
\newcommand{\pdfCiteColor}{magenta}

\usepackage[top=0.9in, bottom=1in, left=0.7in, right=0.7in]{geometry}
\usepackage{amsmath, amssymb, amsthm} 
\usepackage[utf8]{inputenc}
\usepackage[T1]{fontenc}
\usepackage{mathtools}
\usepackage{mathrsfs} 
\usepackage{xfrac} 
\usepackage{dsfont} 
\usepackage{array}
\usepackage{verbatim}
\usepackage{graphicx}
\usepackage{makecell}
\usepackage{mdframed}
\usepackage{enumitem} 
\usepackage{hyperref}
\hypersetup{
	pdftitle={\pdfTitle},
	pdfauthor={\pdfAuthor},
	pdfsubject={\pdfSubject},
	pdfcreationdate={\pdfCreationDate},
	pdfcreator={\pdfCreator},
	pdfkeywords={\pdfKeywords},
	colorlinks=\pdfColorLink,
	linkcolor={\pdfLinkColor},
	urlcolor=\pdfUrlColor,
	citecolor=\pdfCiteColor,
	pdfpagemode=UseOutlines,
}
\usepackage{tikz-cd} 
\usepackage{lipsum}
\usetikzlibrary{matrix,arrows}
\usepackage[english]{babel}
\usepackage{lmodern}
\usepackage{bbm} 
\usepackage[dvipsnames]{xcolor}
\usepackage[most]{tcolorbox}
\usepackage{xparse}

\theoremstyle{plain}
\newtheorem{theorem}{Theorem}[section]
\newtheorem{prop}[theorem]{Proposition}
\newtheorem{lem}[theorem]{Lemma}
\newtheorem{cor}[theorem]{Corollary}

\theoremstyle{definition}
\newtheorem{definition}[theorem]{Definition}

\theoremstyle{remark}
\newtheorem{remark}[theorem]{Remark}

\numberwithin{equation}{section}

\newtheoremstyle{ser}
{8pt}
{8pt}
{\it}
{}
{\sf}
{:}
{6mm}
{}

\theoremstyle{ser}

\newtheoremstyle{serr}
{8pt}
{8pt}
{\normalfont}
{}
{\sf}
{.}
{6mm}
{}

\theoremstyle{serr}

\theoremstyle{ser}

\theoremstyle{ser}
\newtheorem{qn}{Question}

\newtheoremstyle{collabquestion}
  {8pt}
  {8pt}
  {\normalfont}
  {}
  {\sffamily\bfseries\color{blue!70!black}}
  {.}
  {.5em}
  {}

\theoremstyle{collabquestion}
\newtheorem{qninner}{Question}

\definecolor{indraRed}{rgb}{0.593, 0.183, 0.183}
\definecolor{indraPink}{rgb}{0.858, 0.188, 0.478}
\definecolor{indraBlue}{rgb}{0, 0.199, 0.398}
\definecolor{madridBlue}{rgb}{0.199, 0.199, 0.695}
\definecolor{metropolisThemeColor}{rgb}{0.105, 0.214, 0.234}
\definecolor{metropolisBarColor}{rgb}{0.984, 0.0.515, 0.015}
\definecolor{UBCblue}{rgb}{0.04706, 0.13725, 0.26667} 
\definecolor{UBCgrey}{rgb}{0.3686, 0.5255, 0.6235} 

\makeatletter
\def\mathcolor#1#{\@mathcolor{#1}}
\def\@mathcolor#1#2#3{%
	\protect\leavevmode
	\begingroup
	\color#1{#2}#3%
	\endgroup
}
\makeatother

\makeatletter
\def\ps@headings{\ps@empty
  \def\@evenhead{\normalfont\scriptsize\hfil \leftmark{}{}\hfil}%
  \def\@oddhead{\normalfont\scriptsize\hfil \rightmark{}{}\hfil}%
  \let\@mkboth\markboth
  \def\@evenfoot{\normalfont\scriptsize\hfil\thepage\hfil}%
  \def\@oddfoot{\normalfont\scriptsize\hfil\thepage\hfil}%
}
\makeatother

\makeatletter
\def\author@andify{%
  \nxandlist {\unskip ,\penalty-1 \space\ignorespaces}%
    {\unskip {} \@@and~}%
    {\unskip \penalty-2 \space \@@and~}%
}
\makeatother

\definecolor{indraRed}{rgb}{0.593, 0.183, 0.183}
\definecolor{indraPink}{rgb}{0.858, 0.188, 0.478}
\definecolor{indraBlue}{rgb}{0, 0.199, 0.398}
\definecolor{madridBlue}{rgb}{0.199, 0.199, 0.695}
\definecolor{metropolisThemeColor}{rgb}{0.105, 0.214, 0.234}
\definecolor{metropolisBarColor}{rgb}{0.984, 0.0.515, 0.015}
\definecolor{UBCblue}{rgb}{0.04706, 0.13725, 0.26667} 
\definecolor{UBCgrey}{rgb}{0.3686, 0.5255, 0.6235} 

\newcommand{\C}{\mathbb{C}}

\newcommand{\N}{\mathbb{N}}

\NewDocumentCommand{\mn}{m O{\mathbb{C}}}
{\mathbb{M}_{#1}(#2)} 

     \newcommand{\sA}{\mathcal A}		
     \newcommand{\sB}{\mathcal B}

     \newcommand{\sG}{\mathcal G}

     \newcommand{\sJ}{\mathcal J}

     \newcommand{\sN}{\mathcal N}

     \newcommand{\sU}{\mathcal U}

     \newcommand{\sZ}{\mathcal Z}

\NewDocumentCommand{\nor}{g}
  {\sN\IfValueT{#1}{_{#1}}}

\NewDocumentCommand{\unor}{g}
  {\sU\sN\IfValueT{#1}{_{#1}}}

\NewDocumentCommand{\grnor}{g}
  {\sG\sN\IfValueT{#1}{_{#1}}}

\newcommand{\maxprime}{\textsc{Max}$'$\,}

\begin{document}
	
        \title[\ShortTitle]{\MakeUppercase\Title}
	\author{\AuthorOne}
	\address[Ghosh]{\AuthorOneAddr}
	
	\email{\AuthorOneEmail}

    \author{\AuthorTwo}
	\address[Kumar]{\AuthorTwoAddr}
	\email{\AuthorTwoEmail}

	\date{}
	\subjclass{\SubjectClassText}
	
	
	\begin{abstract}
		We give a complete characterization of Cartan inclusions of finite dimensional $C^*$-algebras in terms of their inclusion matrices. More precisely, for a unital inclusion $\mathcal{B}\subseteq\mathcal{A}$ with inclusion matrix $\Lambda=(\Lambda_{ij})$, where Cartan means that $\mathcal{B}$ is a \emph{generalised Cartan subalgebra} of $\mathcal{A}$ in the sense of Exel, we prove that the inclusion is Cartan if and only if
\[
    \sum_i \Lambda_{ij}\leq 1
\]
for every $j$. We call matrices satisfying this condition \emph{multiplicity free}. Thus, our characterization provides a purely combinatorial criterion for determining when a finite dimensional inclusion is Cartan. We further prove that every Cartan inclusion admits a unique conditional expectation from $\mathcal{A}$ onto $\mathcal{B}$. Conversely, we show that, for unital inclusions of finite dimensional $C^*$-algebras, the uniqueness of the conditional expectation is sufficient for the inclusion to be Cartan. Consequently, a unital inclusion $\mathcal{B}\subseteq\mathcal{A}$ of finite dimensional $C^*$-algebras is Cartan if and only if there exists a unique conditional expectation from $\mathcal{A}$ onto $\mathcal{B}$.
	\end{abstract}

	\maketitle%
        \thispagestyle{empty}%


    \section{Introduction}
    Cartan masas are among the most well-studied and important objects in the theory of von Neumann algebras; see, for example, \cite{Dix1954}. Building on the work of Feldman--Moore \cite{Feldman_Moore} and Kumjian \cite{Kumjian_86}, Renault introduced a natural notion of Cartan subalgebras for $C^*$-algebras \cite{Renault2008CartanSI}. In these works, the Cartan subalgebras were \emph{abelian}. In \cite{Exel2011}, Exel introduced a notion of Cartan subalgebras that makes sense for general, not necessarily commutative, $C^*$-subalgebras. He called these subalgebras \emph{generalised Cartan} subalgebras. These, in particular, generalise the earlier notions in which the term ``Cartan'' was used. The main result of \cite{Exel2011}, namely Theorem 14.7, gives a characterisation of generalised Cartan subalgebras of a $C^*$-algebra in terms of the reduced cross-sectional $C^*$-algebra of a suitable Fell bundle over the algebra.

The main purpose of this article is to study the notion of generalised Cartan subalgebras in the setting of finite-dimensional $C^*$-algebras. Since many useful linear algebraic tools are available in finite-dimensions, we seek a more combinatorial characterisation in this setting. To the best of our knowledge, such a study has not yet been undertaken. One of the most powerful tools available for studying inclusions of finite-dimensional $C^*$-algebras is the \emph{inclusion matrix}. Specifically, every unital inclusion of finite-dimensional $C^*$-algebras gives rise to a rectangular matrix with non-negative integer-valued entries, and conversely. Thus, studying such inclusions is equivalent to studying their corresponding inclusion matrices. In this article, we investigate the inclusion matrix associated with a unital inclusion $\sB\subseteq \sA$ where $\sB$ is a generalised Cartan subalgebra of $\sA$. More precisely, we ask the following question:
\begin{qn}
Which inclusion matrix $\Lambda$ can give rise to a unital inclusion $\sB \subseteq \sA$ such that $\sB$ becomes a generalised Cartan subalgebra of $\sA$, and conversely?
\end{qn}

Since a generalised Cartan subalgebra, by definition (see Definition \ref{def:gen_cartan}), comes with a faithful conditional expectation from the parent algebra onto the subalgebra, answering the above question requires us to study the existence and uniqueness of conditional expectations for finite-dimensional inclusions. We address this in \S~\ref{sec:cond_exp}. The existence of conditional expectations is well known in the finite dimensional setting; in fact, there always exists a faithful conditional expectation onto any subalgebra of a finite dimensional $C^*$-algebra. On the other hand, the uniqueness question is more subtle. A well-known sufficient condition for the uniqueness of the conditional expectation from $\sA$ to $\sB$ is that $\sB'\cap \sA \subseteq \sB$ (see \cite[Corollary 1.4.3]{watatani1990}, \cite[IX.4.3]{takesaki02}). In general, however, this condition is not necessary for uniqueness. In Proposition \ref{prop:rel_com_iff_uniq_cond}, we prove that, in the finite dimensional setting, this condition is actually necessary for the uniqueness of the conditional expectation. This leads to the following natural question:
\begin{qn}
Which subalgebras $\sB\subseteq \sA$ exactly satisfy $\sB'\cap \sA \subseteq \sB$, and conversely?
\end{qn}

In searching for an answer to this question, we realised that the subalgebras in question must be what we call \emph{multiplicity free} (see Definition \ref{def:mult_free_subalg}). We devote \S~\ref{sec:mult_free} to the study of these subalgebras. It is observed that every masa in a finite dimensional $C^*$-algebra is multiplicity free (see Proposition \ref{prop:masa_is_mult_free}). These subalgebras are defined in terms of their corresponding inclusion matrices. The inclusion matrix of such a subalgebra has exactly one entry equal to $1$ in each column, with all remaining entries equal to zero. We call such rectangular matrices \emph{multiplicity free} matrices (see Definition \ref{def:multiplicity_free_mat}). At the beginning of \S~\ref{sec:mult_free}, we therefore establish several properties of matrices of this type. Although these results are purely combinatorial in nature, they are essential for the rest of the article. In Proposition \ref{lem:mult_free_iff_rel_com}, we prove that multiplicity free subalgebras are precisely the subalgebras $\sB$ satisfying $\sB'\cap \sA \subseteq \sB$.

Using these results, in \S~\ref{sec:gen_cartan}, we are able to characterise the inclusion matrices corresponding to \emph{Cartan inclusions}. In fact, our main result proves the following:

\noindent \textbf{Theorem} (Theorem \ref{thm:char_gen_cartan}). \textsl{\emph{For a finite-dimensional unital inclusion $\sB \subseteq \sA$ the following statements are equivalent:
\begin{itemize}
\item[(i)] $\sB$ is a generalised Cartan subalgebra of $\sA$.
\item[(vi)] $\sB$ is multiplicity free.
\item[(vii)] there is a unique conditional expectation from $\sA$ onto $\sB$.
\end{itemize}
}
}
This shows that Cartan inclusions are precisely those for which the inclusion matrices are multiplicity free. We conclude our study by listing the total number of Cartan subalgebras for a few given inclusions of finite dimensional $C^*$-algebras.

    \section{Preliminaries}
    \label{sec:prelims}
    The primary purpose of this section is to establish the notation and conventions used throughout the article. We recall a standard result on finite-dimensional $C^*$-algebras, which can be found in standard references such as \cite{davidson_book}. We then introduce some basic definitions concerning normalizers and regular $C^*$-subalgebras.

\noindent\textbf{Notations.} Throughout this article, given $d, p \in\N$ we will use the notation $\mn{d} \otimes \mathbb{I}_p$ to denote the unital $C^*$-subalgebra of $\mn{dp}$ defined by
\[
\mn{d} \otimes \mathbb{I}_p
:=
\left\{
A\otimes \mathbb{I}_p
:=
\begin{pmatrix}
A & & & \\
& A & & \\
& & \ddots & \\
& & & A
\end{pmatrix}
:\;
A\in \mn{d}
\right\}
\subseteq \mn{dp},
\]
where the matrix has $p$ diagonal blocks, each equal to $A$.

We shall also use the following notation for block diagonal subalgebras. Let
\[
\mn{d_j} \otimes \mathbb{I}_{p_j}\subseteq \mn{d_j p_j},
\qquad j = 1, \dots, k.
\]
Then
\[
\bigoplus_{j = 1}^{k} \left( \mn{d_j} \otimes \mathbb{I}_{p_j} \right)
:=
\left\{
\begin{pmatrix}
A_1 & & & \\
& A_2 & & \\
& & \ddots & \\
& & & A_k
\end{pmatrix}
:\;
A_j\in \mn{d_j} \otimes \mathbb{I}_{p_j}\subseteq \mn{d_j p_j},
\; 1 \le j \le k
\right\},
\]
which is a unital $C^*$-subalgebra of $\mn{n}$, where $n=\sum_{j=1}^{k} d_j p_j$.

For a tuple $\Vec{n}=(n_1,\dots,n_l)\in\N^l$, we use the notation
\[
\mn{\Vec{n}}
:=
\bigoplus_{i=1}^l \mn{n_i}.
\]

Let $\Vec{n}\in\N^l$ and $\Vec{d}\in\N^k$ be such that there exists a matrix
\[
\Lambda=[\Lambda_{ij}]\in \mn{l\times k}[\N\cup\{0\}]
\]
satisfying
\[
\Lambda\Vec{d}=\Vec{n}.
\]
Then the pair $(\Vec{d},\Lambda)$ determines a $*$-homomorphism
\begin{align*}
\Phi_\Lambda:\mn{\Vec{d}}&\longrightarrow \mn{\Vec{n}},\\
x&\longmapsto
\left(\Phi_\Lambda^{(1)}(x),\dots,\Phi_\Lambda^{(l)}(x)\right),
\end{align*}
where, for $x=(x_1,\dots,x_k)\in\mn{\Vec{d}}$,
\[
\Phi_\Lambda^{(i)}(x)
:=
\bigoplus_{j=1}^k
\left(x_j\otimes \mathbb{I}_{\Lambda_{ij}}\right),
\qquad i=1,\dots,l.
\]
Here, by convention, a summand with $\Lambda_{ij}=0$ is omitted. Observe that $\Phi_\Lambda$ is unital if and only if every row of $\Lambda$ is non-zero. Likewise, $\Phi_\Lambda$ is injective if and only if every column of $\Lambda$ is non-zero. Henceforth, we assume that $\Lambda$ is such that $\Phi_\Lambda$ is a unital embedding.

We write
\[
\sB(\Vec{d},\Lambda)
:=
\Phi_\Lambda\!\left(\mn{\Vec{d}}\right)
\subseteq
\mn{\Vec{n}}.
\]

It is immediate that $\sB(\Vec{d},\Lambda)$ is a unital $C^*$-subalgebra of $\mn{\Vec{n}}$. It is known that unital $C^*$-subalgebra of $\mn{\Vec{n}}$ arises in this way:

\begin{prop}[\cite{davidson_book}]\label{prop:subalg_of_fin_dim_alg}
Let $\Vec{n}=(n_1,\dots,n_l)\in\N^l$. A subset
$
\sB\subseteq\mn{\Vec{n}}
$
is a unital $C^*$-subalgebra if and only if there exist a unitary $u$ in $\mn{\vec{n}}$, $k\in\N$, a tuple $\Vec{d}\in\N^k$, and a matrix
\[
\Lambda\in\mathbb{M}_{l\times k}(\N\cup\{0\})
\]
satisfying
\[
\Lambda\Vec{d}=\Vec{n},
\]
such that
\[
u\sB u^*=\sB(\Vec{d},\Lambda).
\]
\end{prop}

\begin{definition}
    The matrix $\Lambda$ is called the \emph{inclusion matrix} of the subalgebra $\sB(\Vec{d}, \Lambda)$.
\end{definition}

\noindent\textbf{Convention.} From now on, let $\vec{d}\in \N^k$, $\vec{n}\in \N^l$, and let $\Lambda$ be an $l \times k$ matrix satisfying $\Lambda \vec{d}=\vec{n}$. We will use the phrase
\[
``\text{the subalgebra }\mn{\vec{d}} \subseteq \mn{\vec{n}} \text{ with inclusion matrix } \Lambda"
\]
to mean the unital $C^*$-subalgebra $\sB(\vec{d}, \Lambda)$ of $\mn{\vec{n}}$.

Let $\sB \subset \sA$ be an inclusion clear
of $C^*$-algebras. We define the normalizer and regularity for arbitrary inclusions.

\begin{definition}
The \emph{normalizer} of $\sB$ in $\sA$ is
\[
\nor{\sA}(\sB)
=
\{\, n \in \sA \mid n\sB n^*  \cup n^*\sB n \subseteq \sB \,\}.
\]
It is easily seen that this set is a norm closed $*$-semigroup containing $\sB$.
\end{definition}

\begin{definition}\label{def:regularity}
Let $\mathcal{B}$ be a $C^*$-subalgebra of $\mathcal{A}$. We say that
$\mathcal{B}$ is \emph{regular} in $\mathcal{A}$ if the $C^*$-algebra
generated by $\nor{\mathcal{A}}(\mathcal{B})$ coincides with
$\mathcal{A}$; that is,
\[
C^*(\nor{\mathcal{A}}(\mathcal{B}))=\mathcal{A}.
\]
\end{definition}

We now introduce the definition of a regular matrix, as defined in \cite{BGK2026}. This class of matrices was first introduced in \cite{Bakshi2026}, where the authors referred to them as \emph{normaliser matrices}. We subsequently realised that the two notions are equivalent; see \cite[Proposition 2.9]{BGK2026}.

\begin{definition}\cite[Definition 1.1]{BGK2026}\label{def:regular_mat}
    Let $\Lambda=(\Lambda_{ij})$ be an $l\times k$ matrix with nonnegative integer entries. For each $i=1,\ldots,l$, define the support of the $i$th row by
\[
S_i:=\{\,j:\Lambda_{ij}>0\,\}.
\]
The matrix $\Lambda$ is called \emph{regular} if the following conditions hold:
\begin{enumerate}
    \item[(i)] For each $i=1,\ldots,l$, there exists an integer $p_i\in\mathbb{N}$ such that
    \[
    \Lambda_{ij}=p_i,\qquad \forall\, j\in S_i.
    \]
    Equivalently, all positive entries in the $i$th row are equal.

    \item[(ii)] For each $i=1,\ldots,l$, any two columns indexed by $S_i$ are identical; that is,
    \[
    \vec{\Lambda}_j=\vec{\Lambda}_{j'},\qquad \forall\, j,j'\in S_i.
    \]
\end{enumerate}
\end{definition}
This regular matrix played a crucial role in characterizing general regular inclusions of finite-dimensional $C^*$-algebras. In \cite{BGK2026}, we proved the following:

\begin{theorem}\cite[Theorem 5.11]{BGK2026}\label{thm:regularity_char}
    Let $\vec{n} = (n_1,\dots,n_l)^T\in\mathbb{N}^l$, $\vec{d}=(d_1,\dots,d_k)^T\in\mathbb{N}^k$, and let $\Lambda=[\Lambda_{ij}]_{l\times k}$ satisfy $\Lambda\vec{d}=\vec{n}.$
Then the algebra $\mn{\vec{d}}$ is regular in $\mn{\vec{n}}$ if and only if $\Lambda$ is regular.
\end{theorem}

    \section{Existence and Uniqueness of Conditional Expectations}
    \label{sec:cond_exp}

In this section, we study conditional expectations for inclusions of finite-dimensional $C^*$-algebras. Let $\sB \subseteq \sA$ be a unital inclusion of $C^*$-algebras, not necessarily finite-dimensional. A map $E:\sA \to \sB$ is said to be a conditional expectation if it is a linear, unital, completely positive map satisfying $E(b)=b$ for all $b\in \sB$. It is well known that every such map satisfies the following bimodule property
\[
E(b_1xb_2)=b_1E(x)b_2,
\]
for all $b_1,b_2 \in \sB$ and $x\in \sA$. $E$ is said to be \emph{faithful} if for every $x\neq 0,$ $E(x^*x)>0$. First, we observe in the following lemma that for a finite-dimensional inclusion of $C^*$-algebras, a conditional expectation always exists.

\begin{lem}\label{lem:faithful_exp}
 There exists a faithful conditional expectation from $\mn{\vec{n}}$ onto any of its unital $C^*$-subalgebras. 
\end{lem}
\begin{proof}
    It is easy to see there exists a faithful tracial state on $\mn{\vec{n}}$: indeed consider,
    \begin{align*}
        \tau:\mn{\vec{n}} &\to \C \\   
        (a_1, \dots, a_l)&\mapsto \sum_{i} \omega_i\, \textrm{Tr}(a_i),
    \end{align*}
    for any $\omega_i>0$ such that $\sum_i \omega_i = 1.$ Now if $\sB\subseteq \mn{\vec{n}}$ be any unital $C^*$-subalgebra then by \cite[Proposition V.2.36]{takesaki79}, there exists a faithful conditional expectation $E:\mn{\vec{n}} \to \sB$ such that $\tau \circ E = \tau$.
\end{proof}

\begin{remark}
    We point out that, although the preceding lemma guarantees the existence of \emph{faithful} conditional expectations, a conditional expectation need not be faithful in general. More broadly, even for finite-dimensional unital inclusions, a conditional expectation need not have finite Watatani index; see \cite{watatani1990}. For example, consider the conditional expectation
    \[
        E \colon \mn{2} \ni x \longmapsto x_{11}\otimes \mathbb{I}_2 \in \mathbb{C}\otimes \mathbb{I}_2.
    \]
    It is readily seen that $E$ does not have finite Watatani index. In fact, for a finite-dimensional unital inclusion $\sB \subseteq \sA$, a conditional expectation $E\colon \sA\to\sB$ has finite Watatani index if and only if $E$ is faithful.
\end{remark}

\begin{remark}
    The previous lemma only establishes the existence of a faithful conditional expectation. In general, an inclusion of (finite-dimensional) $C^*$-algebras may admit more than one faithful conditional expectation. Indeed, let $\tau:\mn{\vec{n}}\to \C$ be any faithful tracial state where $\vec{n} \ne 1$. Then there exists a faithful conditional expectation $E_\tau:\mn{\vec{n}} \to \C$ (where $\C \subseteq \mn{\vec{n}}$ has inclusion matrix $\Lambda_{l \times 1} = \vec{n}$) given by
    \[
    E_\tau(a):= \bigoplus_{i = 1}^l \left( \tau(a) \otimes \mathbb{I}_{n_i} \right), \qquad a \in \mn{\vec{n}}.
    \]
    Moreover, the correspondence $\tau \mapsto E_\tau$ is injective. Consequently, distinct faithful tracial states on $\mn{\vec{n}}$ give rise to distinct faithful conditional expectations from $\mn{\vec{n}}$ onto $\C$.
\end{remark}

The next result provides a sufficient condition for the uniqueness of the conditional expectation.

\begin{prop}[{\cite[Corollary 1.4.3]{watatani1990}}]\label{prop:wata_exp}
    Let $\sB \subseteq^E \sA$ be a unital inclusion of $C^*$-algebras satisfying $\sB'\cap \sA \subseteq \sB$. Then $E$ is the only conditional expectation from $\sA$ onto $\sB$.
\end{prop}

\begin{lem}\label{lem:uniq_imply_faithful}
   Let $\sB \subseteq^E \sA$ be a unital inclusion of finite-dimensional $C^*$-algebras. If $E$ is the unique conditional expectation from $\sA$ onto $\sB$, then $E$ is faithful.
\end{lem}
\begin{proof}
 Since $\sA$ is a finite-dimensional $C^*$-algebra, Lemma \ref{lem:faithful_exp} guarantees the existence of a faithful conditional expectation, say $F$. By the uniqueness of the conditional expectation, we have $E=F$. Hence, $E$ is faithful.
\end{proof}

We now characterize the finite-dimensional inclusions of $C^*$-algebras for which the conditional expectation is unique. After posting the first draft of this article, Professor Vrej Zarikian kindly pointed out to us that the following proposition—and, consequently, the implication `(v) $\iff$ (vii)' in Theorem \ref{thm:char_gen_cartan}—was also proved in his recent work \cite[Theorem 2.4-((v) $\iff$ (vi))]{zarikian2026}.

\begin{prop}\label{lem:rel_com_iff_uniq_exp}\label{prop:rel_com_iff_uniq_cond}
    Let $\sB\subseteq\sA$ be a unital inclusion of finite-dimensional $C^*$-algebras. The following conditions are equivalent: 
    \begin{itemize}
        \item[(i)] $\sB' \cap \sA \subseteq \sB$.
        \item[(ii)] there is a unique conditional expectation from $\sA$ onto $\sB$.
    \end{itemize}
\end{prop}
\begin{proof}
($(i) \implies (ii)$). By Lemma \ref{lem:faithful_exp}, there exists a faithful conditional expectation $E:\sA\to\sB$. Since $\sB'\cap \sA\subseteq \sB$, Proposition \ref{prop:wata_exp} implies that $E$ is the unique conditional expectation.

\noindent($(ii) \implies (i)$). Suppose $\sB\subseteq \sA$ has the unique conditional expectation (say) $E$. First note that, $E(\sB^{'}\cap \sA)\subseteq \sZ(\sB)$. Indeed for any $x\in \sB^{'}\cap \sA$ and $b\in B$, we have 
\[
E(x)b= E(xb)= E(bx)=bE(x),
\]
i.e., $E(x)\in \sZ(\sB)$ for all $x\in \sB^{'}\cap \sA$. Now on contrary assume that $\sB' \cap \sA \nsubseteq \sB$. This implies that $\sZ(\sB) \subsetneq \sB^{'}\cap \sA$. Therefore there exists a self-adjoint element $w \in \sB^{'}\cap \sA$ such that $w\notin \sZ (\sB)$. Take $w_{0}:= w- E(w)$. Clearly $w_0$ is a non-zero, self-adjoint element such that $E(w_0)=0$. Also note that $w_0\in \sB^{'}\cap \sA$. For small $t\in \mathbb{R}$, define $c_{t}:= 1_{\sA}+tw_{0} \in \sB^{'}\cap \sA$. Note that for $t$ close enough to $0$, we have $c_t>0$. Since $E$ is faithful then $E(c^{2}_{t})$ is positive, invertible element in the $\sZ(\sB)$. Let $z_{t}:= E(c^{2}_{t})^{-1/2}\in \sZ(\sB)$ and define $c:= c_{t}z_{t}\in \sB^{'}\cap \sA$. Since $c_{t}, z_{t}$ commutes, we get $c^{2}=c^{2}_{t}z^{2}_{t}$. Observe that, 
\[
E(c^2)= E(c^{2}_{t}z^{2}_{t})= E(c^{2}_{t})z^{2}_{t}=1.
\]
Thus using this $c$ define 
\begin{align*}
    E_{c}: \sA &\to \sB\\
    E_{c}(x)&:= E(cxc). 
\end{align*}
Cleary, $E_c$ is a well-defined linear map. Also it is easy to see that $E_c$ is unital, positive and  $\sB$-bimodule map. Hence it is a conditional expectation from $\sA$ onto $\sB$. Then by uniqueness of the conditional expectation, we have $E_{c}(x)= E(x)$ for all $x\in \sA$. Then,
\[
E(x)= E_{c}(x)= E(c_{t}z_{t}xc_{t}z_{t})= z^{2}_{t}E(c_{t}xc_{t})= E(c^{2}_{t})^{-1}E(c_{t}xc_{t}).
\]
This implies that $E(c_{t}xc_{t})= E(c^{2}_{t})E(x)$ for all $x\in \sA$. Recall, $c^{2}_{t}= 1_{\sA}+ 2t w_{0}+ t^{2}w_{0}$. Then 
\begin{equation}\label{eq:equality}
E(c^{2}_{t})E(x)= (1_{\sA}+t^{2}E(w^{2}_{0}))E(x)= E(x)+ t^{2}E(w^{2}_{0})E(x).    
\end{equation}
Also,
\begin{align*}
    E(c_{t}xc_{t})&= E \left((1_{\sA}+tw_{0})x (1_{\sA}+tw_{0})\right) \\
    &= E\left(x+ t(w_0x+xw_0)+t^{2}w_0xw_0 \right)\\
    &= E(x)+tE(w_0x+xw_0)+t^{2}E(w_{0}xw_0)
\end{align*}    
Equating the coefficient of $t$ in equation \ref{eq:equality} and this last equation, we get $E(w_0x+xw_0)=0$ for all $x\in \sA$. Choose $x=w_0$, then $0=E(2w^{2}_{0})= E(w^{2}_{0})=E(w^{*}_{0}w_0)$ and since $E$ is the unique conditional expectation, by Lemma \ref{lem:uniq_imply_faithful}, $E$ is faithful. This implies that $w_0=0$, which is a contradiction since $w_0\neq 0$. This implies that $\sB^{'}\cap \sA \subseteq \sB$.

This completes the proof!
\end{proof}

    \section{Multiplicity Free Subalgebras}
    \label{sec:mult_free}
    
In this section, we introduce the notion of a \emph{multiplicity free} matrix for a rectangular matrix with non-negative integer entries. Using this notion, we define multiplicity free subalgebras and multiplicity free embeddings of finite-dimensional $C^*$-algebras. We conclude this section by characterizing the multiplicity free subalgebras of a given finite-dimensional unital $C^*$-algebra.

\begin{definition}[Multiplicity Free Matrix]\label{def:multiplicity_free_mat}
    Let $\Lambda$ be an $l \times k$ matrix with entries in $\N \cup \{0\}$. We say that $\Lambda$ is \emph{multiplicity free} if
    \[
        \sum_{i=1}^{l} \Lambda_{ij} \leq 1
    \]
    for every $j=1,\dots,k$.
\end{definition}
\begin{remark}\label{rem:multiplicity_implies_reg}
 It can be easily seen that all the nonzero entries in each row of a multiplicity free matrix are equal to $1$, and the corresponding columns are identical. Hence, all the multiplicity free matrices are regular matrices (see Definition \ref{def:regular_mat}).
   \end{remark}

The following proposition shows that, for a square matrix, the multiplicity free matrices are precisely the permutation matrices.

\begin{prop}\label{prop:mult_iff_permutation}
    Let $\Lambda$ be a $k \times k$ matrix with entries in $\N \cup \{0\}$ such that no row or column is identically zero. Then the following are equivalent:
    \begin{itemize}
        \item[(i)] $\Lambda$ is multiplicity free.
        \item[(ii)] $\Lambda$ is a permutation matrix.
        \item[(iii)] There exists $\vec{d} \in \N^k$ such that $\Lambda \vec{d} = \vec{d}$.
    \end{itemize}
\end{prop}
\begin{proof}
     First observe that the implication $(ii) \implies (i)$ is immediate. Likewise, the implication $(ii) \implies (iii)$ is straightforward.
     
    \noindent ($(i) \implies (ii)$). Since no column is entirely zero, (i) implies the sum of every column is exactly 1. Because the entries are non-negative integers, this means every column contains exactly one $1$ and the rest $0$s.

    Since there are $k$ columns, and each column sums to exactly $1$, the sum of all elements in the entire matrix is $k$:
    $$\text{Total Sum} = \sum_{j=1}^k \left( \sum_{i=1}^{k} \Lambda_{ij} \right) = \sum_{j=1}^k 1 = k.$$

    Let $r_i$ be the sum of the entries in the $i$-th row. We are given that no row is entirely zero. Therefore, each row sum must be at least 1:$$r_i \geq 1 \quad \text{for all } i \in \{1, \dots, k\}.$$
    We also know that summing all the row sums must equal the total sum of the matrix, which we found is $k$:$$\sum_{i=1}^k r_i = k.$$
    We are summing $k$ positive integers ($r_1, r_2, \dots, r_k$), each of which is $\geq 1$, and their total is exactly $k$. This is possible if every single row sum is exactly 1, i.e. $$r_i = 1 \quad \text{for all } i \in \{1, \dots, k\}.$$
    Since all the entries are in $\mathbb{N} \cup \{0\}$, every row and every column contains exactly one $1$ and $(k-1)$ $0$s. This is the exact definition of a permutation matrix.

    \noindent ($(iii) \implies (i)$). Assume there exists $\vec{d} \in \mathbb{N}^k$ such that $\Lambda \vec{d} = \vec{d}$. Therefore we get for all $i,\, d_i = \sum_{j=1}^k \Lambda_{ij} d_j$.
    Summing all $k$ components gives:
    \begin{equation}\label{eq:component_sum}
     \sum_{i=1}^k d_i = \sum_{i=1}^k \left( \sum_{j=1}^k \Lambda_{ij} d_j \right)=\sum_{j=1}^k d_j \left( \sum_{i=1}^k \Lambda_{ij} \right).   
    \end{equation}
    Let $c_j = \sum_{i=1}^k \Lambda_{ij}$ be the sum of the $j$-th column. We are given that no column is entirely zero, and all entries are in $\mathbb{N} \cup \{0\}$. Therefore, every column sum must be at least 1 (i.e., $c_j \ge 1$). Substituting $c_j$ back into equation \ref{eq:component_sum}, we get
    $$\sum_{j=1}^k d_j (c_j - 1) = 0.$$
    This implies that $c_j = 1$ for all $j = 1, \dots, k$. Thus the sum of every column is exactly 1, we trivially have, $\sum_{i=1}^k \Lambda_{ij} \leq 1 \,\,\, \text{for all } j$, that is, $\Lambda$ is multiplicity free.
\end{proof}

\begin{prop}\label{prop:mul_embed}
    Let $\vec{d} \in \N^k$ and $\vec{n} \in \mathbb{N}^l$ be fixed vectors. Let $S_d = \sum_{j=1}^k d_j$ and $S_n = \sum_{i=1}^l n_i$. Suppose there exists an $l \times k$ matrix $\Lambda$ with entries in $\mathbb{N} \cup \{0\}$, having no identically zero columns, such that $\Lambda \vec{d} = \vec{n}$.
    \begin{enumerate}
    \item[(i)] If $\Lambda$ is multiplicity free, then $S_d = S_n$.
    \item[(ii)] If $\Lambda$ is not multiplicity free, then $S_d < S_n$.
\end{enumerate}
Consequently, for a given fixed pair of vectors $(\vec{d}, \vec{n})$, it is impossible to find two such matrices $\Lambda_1$ and $\Lambda_2$ where one is multiplicity free and the other is not.
\end{prop}
\begin{proof}
    \noindent (i). Because $\Lambda$ is multiplicity free and has no zero columns, the sum of every single column must be exactly $1$. If we sum all the entries of $\vec{n} = \Lambda \vec{d}$, we get:
    $$S_n = \sum_{i=1}^l n_i = \sum_{i=1}^l \left( \sum_{j=1}^k \Lambda_{ij} d_j \right) = \sum_{j=1}^k d_j \left( \sum_{i=1}^l \Lambda_{ij} \right)= \sum_{j=1}^k d_j (1) = S_d.$$

    \noindent (ii). Because $\Lambda$ is not multiplicity free but still has no zero columns, every column sums to at least $1$, and at least one column (let's say column $x$) sums to $2$ or more. If we do the exact same summation for $\vec{n} = \Lambda \vec{d}$, we get:
    $$S_n = \sum_{j=1}^k d_j \left( \sum_{i=1}^l \Lambda_{ij} \right) \ge \left( \sum_{j \neq x} d_j (1) \right) + d_x (2) = \left( \sum_{j=1}^k d_j \right) + d_x = S_d + d_x>S_d.$$
\end{proof}

It is natural to ask, for a fixed pair of vectors $(\vec{d},\vec{n})\in \N^k \times \N^l$, how many multiplicity free matrices $\Lambda$ satisfy $\Lambda \vec{d}=\vec{n}$.
The following proposition provides a recipe for finding this number.

\begin{prop}\label{prop:num_of_mult_free}
    Let $\vec{d}=(d_1, \dots, d_k) \in \N^k$ and $\vec{n} =(n_1, \dots, n_l)\in \mathbb{N}^l$ be fixed vectors. The number of multiplicity free matrices $\Lambda \in \mathbb{M}_{l\times k}(\mathbb{N} \cup \{0\})$ with no-zero columns satisfying $\Lambda \vec{d} = \vec{n}$ is equal to the coefficient of $x_1^{n_1} x_2^{n_2} \cdots x_l^{n_l}$ in the expansion of the following polynomial,
    $$P_{\vec{n}}^{\vec{d}}(x_1, x_2, \dots, x_l) := \prod_{j=1}^k \left( x_1^{d_j} + x_2^{d_j} + \dots + x_l^{d_j} \right).$$
\end{prop}
\begin{proof}
    Because $\Lambda$ is multiplicity free and has no identically zero columns, the sum of each column must be exactly 1. Since the entries are non-negative integers, every column $j$ contains exactly one $1$ and $(l-1)$ $0$s.

    This means each matrix $\Lambda$ uniquely defines a function (an assignment) $f: \{1, \dots, k\} \to \{1, \dots, l\}$, where: $f(j) = i$ if and only if $\Lambda_{ij} = 1$.

    The condition that $\Lambda \vec{d} = \vec{n}$ means that for every row $i \in \{1, \dots, l\}$, the dot product of the $i$-th row with $\vec{d}$ equals $n_i$:$$\sum_{j=1}^k \Lambda_{ij} d_j = n_i$$Since $\Lambda_{ij} = 1$ exactly when $f(j) = i$ (and is $0$ otherwise), we can rewrite this sum as:
    \begin{equation}
    \label{eqn:sum_cond}
        \sum_{j \in f^{-1}(i)} d_j = n_i.
    \end{equation}

    Thus, finding a valid matrix $\Lambda$ is strictly equivalent to finding an assignment function $f$ such that the sum of the elements mapped to $i$ equals $n_i$ for all $i$.
    Consider the polynomial $$P_{\vec{n}}^{\vec{d}}(x_1, x_2, \dots, x_l) = \prod_{j=1}^k \left( \sum_{i=1}^l x_i^{d_j} \right).$$
    By the distributive law of algebra, expanding this product requires selecting exactly one term from each of the $k$ factors and multiplying them together.

    A choice of one term from each factor $j$ corresponds perfectly to choosing an assignment function $f(j) = i$. Therefore, we can write the fully expanded (but unsimplified) polynomial as a sum over all possible functions $f$:
    $$P_{\vec{n}}^{\vec{d}}(x_1, \dots, x_l) = \sum_{f} \left( \prod_{j=1}^k x_{f(j)}^{d_j} \right)$$

    For a specific function $f$, let's group the $x$ variables in the product by their subscripts. If multiple items are assigned to the same $x_i$, their exponents are added together:$$\prod_{j=1}^k x_{f(j)}^{d_j} = x_1^{\sum_{j \in f^{-1}(1)} d_j} \cdot x_2^{\sum_{j \in f^{-1}(2)} d_j} \cdots x_l^{\sum_{j \in f^{-1}(l)} d_j}$$Substituting this back into our polynomial expansion gives:
    \begin{equation}
    \label{eqn:poly_sum}
        P_{\vec{n}}^{\vec{d}}(x_1, \dots, x_l) = \sum_{f} \left( \prod_{i=1}^l x_i^{\sum_{j \in f^{-1}(i)} d_j} \right)
    \end{equation}

    As seen above a valid matrix corresponds to a function $f$ where $\sum_{j \in f^{-1}(i)} d_j = n_i$ for every $i$.
    
    When we look at the polynomial sum in (\ref{eqn:poly_sum}), an assignment $f$ will produce the specific algebraic term $x_1^{n_1} x_2^{n_2} \cdots x_l^{n_l}$ if and only if it satisfies that exact sum condition (\ref{eqn:sum_cond}). 

    When we gather like terms to simplify the polynomial, every valid assignment $f$ adds exactly $1$ to the coefficient of $x_1^{n_1} x_2^{n_2} \cdots x_l^{n_l}$. All invalid assignments $f$ produce terms with different exponents and do not contribute to this coefficient.

    Therefore, the final coefficient of $x_1^{n_1} x_2^{n_2} \cdots x_l^{n_l}$ is exactly the number of assignment functions $f$ satisfying the condition, which is precisely the number of multiplicity-free matrices $\Lambda$ mapping $\vec{d}$ to $\vec{n}$
\end{proof}

\begin{theorem}\label{thm:mult_free_invariance}
    Let $\vec{n} \in \mathbb{N}^l$ be fixed. Suppose there exist
    $\vec{d} \in \mathbb{N}^k$, $\Lambda_1 \in \mathbb{M}_{l \times k}(\N \cup \{0\})$,
    $\vec{c} \in \mathbb{N}^r$, and $\Lambda_2 \in \mathbb{M}_{l \times r}(\N \cup \{0\})$
    such that
    \[
        \Lambda_1 \vec{d} = \vec{n} = \Lambda_2 \vec{c}
    \]
    and
    \[
        \Phi_{\Lambda_1}(\mn{\vec{d}})
        =
        \Phi_{\Lambda_2}(\mn{\vec{c}}).
    \]
    If $\Lambda_1$ is multiplicity-free, then $\Lambda_2$ is also
    multiplicity-free.
\end{theorem}
\begin{proof}
    Let $C_1 = \Phi_{\Lambda_1}\left(\mathbb{M}_{\vec{d}}(\mathbb{C})\right)$ and $C_2 = \Phi_{\Lambda_2}\left(\mathbb{M}_{\vec{c}}(\mathbb{C})\right)$. By hypothesis, these are the exact same $C^*$-subalgebra of $\mathbb{M}_{\vec{n}}(\mathbb{C})$, so $C_1 = C_2 = C$. Hence,
    $$\sZ(C_1) = \sZ(C_2) = \sZ(C).$$
    The center of the finite-dimensional algebra $\mathbb{M}_{\vec{d}}(\mathbb{C}) = \bigoplus_{j=1}^k \mathbb{M}_{d_j}(\mathbb{C})$ is isomorphic to $\mathbb{C}^k$, and its minimal projections are given by $e_j = (0, \dots, \mathbb{I}_{d_j}, \dots, 0)$ for $j = 1, \dots, k$. Because $\Phi_{\Lambda_1}$ is an injective $*$-homomorphism (an embedding), it maps the center of $\mathbb{M}_{\vec{d}}(\mathbb{C})$ isomorphically onto $\sZ(C_1)$. Thus, the minimal projections of $\sZ(C_1)$ are exactly the images of the minimal projections of $\mathbb{M}_{\vec{d}}(\mathbb{C})$ i.e.,
    $$E_j = \Phi_{\Lambda_1}(e_j), \quad j = 1, \dots, k.$$
    By the exact same reasoning, the minimal projections of $\sZ(C_2)$ are$$F_m = \Phi_{\Lambda_2}(f_m), \quad m = 1, \dots, r,$$where $f_m = (0, \dots, \mathbb{I}_{c_m}, \dots, 0)$ are the minimal central projections of $\mathbb{M}_{\vec{c}}(\mathbb{C})$. Since $\sZ(C_1) = \sZ(C_2)$, they must have the exact same set of minimal projections in $\mathbb{M}_{\vec{n}}(\mathbb{C})$. This implies that $k = r$, and there must exist a permutation $\sigma \in S_k$ such that for all $j \in \{1, \dots, k\}$:$$E_j = F_{\sigma(j)}.$$ 
    Also, the $i$-th component of $E_j$ in $\mathbb{M}_{n_i}(\mathbb{C})$ is,
    $$E_j^{(i)} = \Phi_{\Lambda_1}^{(i)}(e_j) = e_j \otimes \mathbb{I}_{(\Lambda_1)_{ij}}.$$This is a projection in $\mathbb{M}_{n_i}(\mathbb{C})$ of rank $d_j (\Lambda_1)_{ij}$. Similarly, the $i$-th component of $F_{\sigma(j)}$ has rank $c_{\sigma(j)} (\Lambda_2)_{i, \sigma(j)}$. Since $E_j = F_{\sigma(j)}$ exactly as matrices, their ranks in each direct summand must be equal, i.e.
    $$d_j (\Lambda_1)_{ij} = c_{\sigma(j)} (\Lambda_2)_{i, \sigma(j)} \quad \text{for all } i=1, \dots, l.$$
    Furthermore, consider the cut-down algebra $E_j C E_j$. Note that, $E_j C_1 E_j \cong \mathbb{M}_{d_j}(\mathbb{C})$. On the other hand, utilizing $\Phi_{\Lambda_2}$, $F_{\sigma(j)} C_2 F_{\sigma(j)} \cong \mathbb{M}_{c_{\sigma(j)}}(\mathbb{C})$. Because $C_1 = C_2$ and $E_j = F_{\sigma(j)}$, the cut-down algebras are strictly identical, which means $\mathbb{M}_{d_j}(\mathbb{C}) \cong \mathbb{M}_{c_{\sigma(j)}}(\mathbb{C})$. Thus:$$d_j = c_{\sigma(j)}.$$
    Since $d_j \in \mathbb{N}$, We get $$(\Lambda_1)_{ij} = (\Lambda_2)_{i, \sigma(j)} \quad \text{for all } i=1,\dots,l \text{ and } j=1,\dots,k.$$ This implies that the matrix $\Lambda_2$ is exactly the matrix $\Lambda_1$ with its columns permuted by the permutation $\sigma$. We are given that $\Lambda_1$ is multiplicity free, i.e., 
    $$\sum_{i=1}^l (\Lambda_1)_{ij} \leq 1 \quad \text{for all } j=1,\dots,k.$$Let $m \in \{1,\dots,k\}$ be an arbitrary column index for $\Lambda_2$. There exists $j$ such that $m = \sigma(j)$. Summing the entries of the $m$-th column of $\Lambda_2$, we get:$$\sum_{i=1}^l (\Lambda_2)_{im} = \sum_{i=1}^l (\Lambda_2)_{i, \sigma(j)} = \sum_{i=1}^l (\Lambda_1)_{ij} \leq 1.$$Therefore, $\Lambda_2$ is also multiplicity free.
\end{proof}

We now introduce the following two notions associated with multiplicity free matrices.

\begin{definition}[Multiplicity Free Subalgebra]
\label{def:mult_free_subalg}

A unital $C^*$-subalgebra $\sB$ of $\mn{\vec{n}}$ is said to be \emph{multiplicity free} if there exists $u \in \sU(\mn{\vec{n}})$ such that
\[
u\sB u^* = \Phi_{\Lambda}(\mn{\vec{d}})
\]
for some $\vec{d} \in \N^k$ and a multiplicity free $l \times k$ matrix $\Lambda$.
\end{definition}

\begin{remark}
    Note that for a subalgebra $\sB$ of $\mn{\vec{n}}$, if there exists a unitary $u \in \mn{\vec{n}}$ and two matrices $\Lambda_1$ and $\Lambda_2$ such that
\[
\operatorname{im}(\Phi_{\Lambda_1})
= u\sB u^*
= \operatorname{im}(\Phi_{\Lambda_2}),
\]
then, by Theorem \ref{thm:mult_free_invariance}, if either $\Lambda_1$ or $\Lambda_2$ is multiplicity free, then the other is multiplicity free as well.
\end{remark}

\begin{lem}\label{lem:mult-free-is-regular}
Every multiplicity free subalgebra $\sB$ of a unital $C^*$-algebra $\sA$ is regular.    
\end{lem}
\begin{proof}
   As $\sB$ is multiplicity free, by definition there exists $u \in \sU(\mn{\vec{n}})$ such that
\[
u\sB u^*=\Phi_{\Lambda}(\mn{\vec{d}})
\]
for some $\vec{d}\in\N^k$ and a multiplicity free $l\times k$ matrix $\Lambda$. By Remark \ref{rem:multiplicity_implies_reg}, $\Lambda$ is regular. Hence, Theorem \ref{thm:regularity_char} implies that $\sB$ is regular.

\end{proof}
\begin{definition}[Multiplicity Free Embedding]\label{def:mult_free_embed}
    Let $\sA$ and $\sB$ be finite-dimensional abstract $C^*$-algebras. An embedding $\Phi:\sB \to \sA$ is called \emph{multiplicity free} if the image $\Phi(\sB)$ is a multiplicity free subalgebra of $\sA$.
\end{definition}

\begin{remark}
    It is easy to see that for $d\in \N$ and $\vec{n}=(n_1, \dots, n_l)\in \N^l$, there exists a unital embedding of $\mn{d}$ into $\mn{\vec{n}}$ if and only if $d \mid n_i$ for every $i=1,\dots,l$. However, it is readily seen that none of these embeddings is multiplicity free unless $\vec{n}=d\in\N$.
\end{remark}

From the Defintion \ref{def:mult_free_embed} and Proposition \ref{prop:mult_iff_permutation}, we get the following corollary:

\begin{cor}
    Every unital embedding of $\mn{\vec{d}}$ into $\mn{\vec{d}}$ is multiplicity free. Moreover, the total number of such embeddings is precisely
    \[
        \prod_{t} m_t!,
    \]
    where $m_t$ denotes the multiplicity of the value $t$ among the entries of $\vec{d}$. In particular, if all the entries of $\vec{d}$ are equal, then the total number of such embeddings is $k!$, where $k$ is the length of $\vec{d}$.
\end{cor}

\begin{remark}
    Suppose $\sB$ and $\sA$ are abstract finite dimensional $C^*$-algebras. It was observed in \cite[Remark 4.8]{BGK2026} that regularity can depend strongly on the particular embedding of the child $C^*$-algebra into the parent algebra. In other words, there may exist two distinct unital embeddings of $\sB$ into $\sA$ such that one is regular while the other is not. The following result shows that this phenomenon does not occur for multiplicity free subalgebras.
\end{remark}

\begin{theorem}
    Suppose there exists a multiplicity free embedding of an abstract finite dimensional $C^*$-algebra $\sB$ into another finite dimensional abstract $C^*$-algebra $\sA$. Then every embedding of $\sB$ into $\sA$ is multiplicity free.
\end{theorem}
\begin{proof}
    Follows from Proposition \ref{prop:mul_embed}.
\end{proof}

\begin{prop}\label{prop:masa_is_mult_free}
    Every masa in $\mn{\vec{n}}$ is multiplicity free.
\end{prop}
\begin{proof}
    Indeed, every masa of $\mn{\vec{n}}$ is the direct sum of masas of the component algebras $\mn{n_i}$. Moreover, up to unitary conjugacy, $\mn{n_i}$ admits a unique masa, namely
\[
\Delta_1^{(i)}:= \bigoplus_{j=1}^{n_i} \left( \C \otimes \mathbb{I}_1 \right) \subseteq \mn{n_i}.
\]
Consequently, up to unitary conjugacy, the unique masa of $\mn{\vec{n}}$ is
\[
\Delta_1^{\vec{n}}:=\bigoplus_{i=1}^l \Delta_1^{(i)}.
\]

We now compute the inclusion matrix of the inclusion
\[
\Delta_1^{\vec{n}} \subseteq \mn{\vec{n}}.
\]
To this end, observe that $\Delta_1^{\vec{n}}=\operatorname{im}(\Phi_\Lambda)$, where
\[
\Phi_\Lambda:\C^{n_1+\cdots+n_l}\longrightarrow \mn{\vec{n}},
\]
and the inclusion matrix $\Lambda$ together with the dimension vector $\vec{d}$ are given by
\[
\Lambda:=
\begin{pmatrix}
\underbrace{1\ \cdots\ 1}_{n_1} &
\underbrace{0\ \cdots\ 0}_{n_2} &
\cdots &
\underbrace{0\ \cdots\ 0}_{n_l}
\\
\underbrace{0\ \cdots\ 0}_{n_1} &
\underbrace{1\ \cdots\ 1}_{n_2} &
\cdots &
\underbrace{0\ \cdots\ 0}_{n_l}
\\
\vdots & \vdots & \ddots & \vdots
\\
\underbrace{0\ \cdots\ 0}_{n_1} &
\underbrace{0\ \cdots\ 0}_{n_2} &
\cdots &
\underbrace{1\ \cdots\ 1}_{n_l}
\end{pmatrix}_{\,l\times (n_1+\cdots+n_l)}
\qquad
\vec{d}:=
\begin{pmatrix}
1\\
\vdots\\
1\\
\vdots\\
1
\end{pmatrix}_{(n_1+\cdots+n_l)\times 1}
\]

It is evident that $\Lambda$ is multiplicity free. Hence, $\Delta_1^{\vec{n}}$ is a multiplicity free subalgebra of $\mn{\vec{n}}$.
\end{proof}

\begin{lem}\label{lem:rel_comm_invariance}
    Let $\sB \subseteq \sA$ be a unital inclusion of (not necessarily finite-dimensional) $C^*$-algebras, and let $\sigma \in \textrm{Aut}(\sA)$. Then  $\sB ' \cap \sA \subseteq \sB$ if and only if $\sB'\cap \sA = \sZ ( \sB)$ if and only if $\sigma(\sB)' \cap \sA = \sZ(\sigma(\sB))$.
\end{lem}
\begin{proof}
    This follows from the two observations: $\sigma(\sB)'\cap \sA = \sigma(\sB' \cap \sA)$ and $\sigma(\sZ(\sB)) = \sZ(\sigma(\sB))$.
\end{proof}

\begin{lem}\label{lem:rel_com}
    Consider the unital subalgebra $\mn{\vec{d}}$ of $\mn{\vec{n}}$ with the inclusion matrix $\Lambda$. Then
    \[
    \mn{\vec{d}}' \cap \mn{\vec{n}} = \bigoplus_{i=1}^{l} \left( \bigoplus_{j=1}^k \left( \mathbb{I}_{d_j} \otimes \mn{\Lambda_{ij}} \right) \right).
    \]
\end{lem}
\begin{proof} Observe that for $pd = n$,
    \begin{equation}\label{eqn:b_comm}
         \left( \mn{d} \otimes \mathbb{I}_{p} \right)' \cap \mn{n}
        =
            \mathbb{I}_{d} \otimes \mn{p}.
    \end{equation}
    Indeed (recall notations from \S~\ref{sec:prelims}),
    \begin{align*}
    \mn{\vec{d}}' \cap \mn{\vec{n}}
    &= \left\{ (x_1, \dots, x_l) \in \mn{\vec{n}} :
    x_i \in \Phi_\Lambda^{(i)}(\mn{\vec{d}})' \cap \mn{n_i}
    \text{ for all } i \right\} \\
    &= \left\{ (x_1, \dots, x_l) \in \mn{\vec{n}} :
    x_i \in \bigoplus_{j=1}^k
    \left( \mathbb{I}_{d_j} \otimes \mn{p_j} \right)
    \text{ for all } i \right\},
    \qquad \text{by \eqref{eqn:b_comm}}.
    \end{align*}
\end{proof}

We are now ready to prove the main result of this section, which characterizes the multiplicity free subalgebras of finite-dimensional $C^*$-algebras.
\begin{prop}\label{lem:mult_free_iff_rel_com}
    Let $\sB\subseteq\sA$ be a unital inclusion of finite-dimensional $C^*$-algebras. Then the following are equivalent: 
    \begin{itemize}
        \item[(i)] $\sB' \cap \sA \subseteq \sB$.
        \item[(ii)] $\sB$ is multiplicity free.
    \end{itemize}
\end{prop}
\begin{proof}
    Without loss of generality assume $\sB = \mn{\vec{d}}$ and $\sA= \mn{\vec{n}}$ where $\vec{d}= (d_1, \dots, d_k)$ and $\vec{n} = (n_1, \dots, n_l)$ and $\Phi_\Lambda:\mn{\vec{d}}\to \mn{\vec{n}}$ is an embedding with respect to the inclusion matrix $\Lambda$. 

    By definition, an element $Z \in \sB$ is completely determined by a tuple $x = (x_1, \dots, x_k) \in \mn{\vec{d}}$, and takes the explicit block-diagonal form:
    $$Z = \Phi_\Lambda(x) = \bigoplus_{i=1}^l \bigoplus_{j=1}^k (x_j \otimes \mathbb{I}_{\Lambda_{ij}})$$
    From Lemma \ref{lem:rel_com}, the relative commutant of $\sB$ in $\mn{\vec{n}}$ is known to be:
    $$\sB' \cap \mn{\vec{n}} = \bigoplus_{i=1}^l \bigoplus_{j=1}^k (\mathbb{I}_{d_j} \otimes \mn{\Lambda_{ij}})$$

    Thus, an arbitrary element $Y \in \sB' \cap \mn{\vec{n}}$ is determined by a collection of matrices $\{y_{ij}\}_{i=1, j=1}^{l, k}$ where $y_{ij} \in \mn{\Lambda_{ij}}$, taking the form:$$Y = \bigoplus_{i=1}^l \bigoplus_{j=1}^k (\mathbb{I}_{d_j} \otimes y_{ij})$$We will prove the forward and backward directions separately.

    \vspace{0.3cm}
    \noindent($(ii)\implies (i)$) Assume $\sum_{i=1}^l \Lambda_{ij} \le 1$ for all $j=1, \dots, k$. We want to show that any $Y \in \sB' \cap \mn{\vec{n}}$ belongs to $\sB$. 
    
    Let $Y = \bigoplus_{i=1}^l \bigoplus_{j=1}^k (\mathbb{I}_{d_j} \otimes y_{ij}) \in B' \cap \mn{\vec{n}}$. 
    
    Fix an index $j \in \{1, \dots, k\}$. Since $\Phi_\Lambda$ is an embedding, no row or column of $\Lambda$ is zero. Thus the condition $\sum_{i=1}^l \Lambda_{ij} \le 1$ implies:

    There is exactly one index $i_0$ such that $\Lambda_{i_0 j} = 1$, and $\Lambda_{ij} = 0$ for all $i \neq i_0$.
    Because $\Lambda_{i_0 j} = 1$, the matrix $y_{i_0 j}$ belongs to $\mn{1} \cong \mathbb{C}$. Thus, $y_{i_0 j}$ is just a scalar $\lambda_j$.
    We choose $x_j = \lambda_j \mathbb{I}_{d_j} \in \mn{d_j}$.

    Now consider the image of $x = (x_1, \dots, x_k)$ under $\Phi_\Lambda$. For the non-zero block at $(i_0, j)$, we have:
    $$x_j \otimes \mathbb{I}_{\Lambda_{i_0 j}} = (\lambda_j \mathbb{I}_{d_j}) \otimes \mathbb{I}_1 = \mathbb{I}_{d_j} \otimes \lambda_j = \mathbb{I}_{d_j} \otimes y_{i_0 j}$$
    For all other $i \neq i_0$, both sides are $0$. Since this holds for every $j$, we have $\Phi_\Lambda(x) = Y$. Therefore, $Y \in \sB$, proving $\sB' \cap \mn{\vec{n}} \subseteq \sB$.

    \vspace{0.3cm}
    \noindent($(i) \implies (ii)$) Assume $\sB' \cap \mn{\vec{n}} \subseteq \sB$. We want to show $\sum_{i=1}^l \Lambda_{ij} \le 1$ for all $j$.
    We proceed by contrapositive. Suppose there exists some column index $j_0$ such that $\sum_{i=1}^l \Lambda_{i j_0} \ge 2$. There are two distinct ways this can happen, and we will construct a counterexample $Y \in \sB' \cap \mn{\vec{n}}$ that is not in $\sB$ for both.

    \noindent{\bf Case 1:} A single entry is strictly greater than 1 ($\Lambda_{i_0 j_0} \ge 2$ for some $i_0$).

    Construct $Y \in B' \cap \mn{\vec{n}}$ by setting $y_{i_0 j_0} \in \mn{\Lambda_{i_0 j_0}}$ to be any non-scalar matrix (e.g., a matrix with a non-zero off-diagonal entry), and setting all other $y_{ij} = 0$.
    If $Y \in \sB$, there must exist some $x = (x_1, \dots, x_k) \in \mn{\vec{d}}$ such that $\Phi_\Lambda(x) = Y$.Looking specifically at the $(i_0, j_0)$ block, we must have:$$x_{j_0} \otimes \mathbb{I}_{\Lambda_{i_0 j_0}} = \mathbb{I}_{d_{j_0}} \otimes y_{i_0 j_0}$$
    Taking the commutant of both sides with respect to $M_{d_{j_0}} \otimes M_{\Lambda_{i_0 j_0}}$, the left side is central in $M_{\Lambda_{i_0 j_0}}$, meaning it forces $y_{i_0 j_0}$ to be a scalar multiple of $\mathbb{I}_{\Lambda_{i_0 j_0}}$. This contradicts our choice of $y_{i_0 j_0}$ as a non-scalar matrix. Thus, $Y \notin \sB$.

    \noindent{\bf Case 2:} Two different entries in the same column are 1 ($\Lambda_{i_1 j_0} = 1$ and $\Lambda_{i_2 j_0} \ge 1$ for $i_1 \neq i_2$).
    Construct $Y \in B' \cap \mn{\vec{n}}$ by setting:
    \begin{itemize}
        \item[1.] $y_{i_1 j_0} = 1 \in \mn{1}$
        \item[2.] $y_{i_2 j_0} = 0 \in \mn{\Lambda_{i_2 j_0}}$
        \item[3.] All other $y_{ij} = 0$.
    \end{itemize}
    If $Y \in \sB$, there must exist some $x \in \mn{\vec{d}}$ such that $\Phi_\Lambda(x) = Y$.Looking at the $(i_1, j_0)$ block:
    $$x_{j_0} \otimes \mathbb{I}_1 = \mathbb{I}_{d_{j_0}} \otimes 1 \implies x_{j_0} = \mathbb{I}_{d_{j_0}}$$
    Looking at the $(i_2, j_0)$ block simultaneously:
    $$x_{j_0} \otimes \mathbb{I}_{\Lambda_{i_2 j_0}} = \mathbb{I}_{d_{j_0}} \otimes 0 \implies x_{j_0} = 0$$
    This yields $\mathbb{I}_{d_{j_0}} = 0$, which is a contradiction (assuming the dimension $d_{j_0} \ge 1$). Thus, $Y \notin B$.

    In both cases, we found an element in the relative commutant that is not in $\sB$. Therefore, if $\sB' \cap \mn{\vec{n}} \subseteq \sB$, it must be true that $\sum_{i=1}^l \Lambda_{ij} \le 1$ for all $j$.

    This completes the proof.
\end{proof}

\begin{cor}
    Every multiplicity free abelian subalgebra of $\mn{\vec{n}}$ is a masa.
\end{cor}

    \section{Generalised Cartan Subalgebras}
    \label{sec:gen_cartan}
    
In this section, we characterize generalized Cartan subalgebras of finite-dimensional $C^*$-algebras. We begin by recalling some basic definitions from \cite{Exel2011}, including the notion of a generalized Cartan subalgebra. We then introduce the notion of Cartan embeddings for finite-dimensional $C^*$-algebras. Finally, we conclude the section by giving explicit counts of generalized Cartan embeddings for some fixed choices of $\vec{d}$ and $\vec{n}$.

\begin{definition}[{\cite[Definition 9.2]{Exel2011}}]
  Let $\sB\subseteq \sA$ be a inclusion of $C^*$-algebras. A virtual commutant of $\sB$ in $\sA$ is an $\sA$-valued  linear map $\varphi$ defined on a closed two sided ideal $\sJ$ of $\sB$ such that,
  \begin{itemize}
      \item[(i)] $\varphi(bx)= b\varphi (x)$
      \item[(ii)]$\varphi(xb)= \varphi (x)b$
  \end{itemize}
  for all $x\in \sJ$ and $b\in \sB$.
  \end{definition}
 \begin{definition}[{\cite[Definition 9.6]{Exel2011}}]
     A subalgebra $\sB$ satisfies the property \maxprime, if the range of any virtual commutant of $\sB$ in $\sA$ is contained in $\sB$.  
 \end{definition}

\begin{definition}[{\cite[Definition 12.1]{Exel2011}}]\label{def:gen_cartan}
  Let $\sA$ be a $C^*$-algebra and $\sB\subseteq \sA$ be an inclusion of $C^*$-subalgebra. We say that $\sB$ is a \emph{generalized Cartan subalgebra} of $\sA$ if it satisfies the following properties:
  \begin{itemize}
      \item[(i)] $\sB$ contains an approximate unit for $\sA$.
      \item[(ii)] $\sB$ satisfies the \maxprime.
      \item[(iii)] $\sB$ is regular in $\sA$.
      \item[(iv)] There exists a faithful conditional expectation $E: \sA \to \sB$. 
  \end{itemize}
\end{definition}

\begin{definition}[Cartan Embedding]
    Let $\sA$ and $\sB$ be finite dimensional abstract $C^*$-algebras. An embedding $\Phi:\sB \to \sA$ is called \emph{Cartan} if the image $\Phi(\sB)$ is a generalised Cartan subalgebra of $\sA$.
\end{definition}

\begin{lem}\label{lem:autmorph_pres_maxprime}
Let $\sB \subseteq \sA$ be a unital inclusion of (not necessarily finite-dimensional) $C^*$-algebras, and let $\sigma \in \textrm{Aut}(\sA)$. Then $\sB$ has \maxprime in $\sA$ if and only if $\sigma(\sB)$ has \maxprime in $\sA$.
\end{lem}
\begin{proof}
    The proof follows from the observation that $(\sJ, \varphi)$ is a virtual commutant of $\sigma(\sB)$ if and only if $(\sigma^{-1}(\sJ),\, \sigma^{-1} \circ \varphi \circ \sigma)$ is a virtual commutant of $\sB$.
\end{proof}

\begin{prop}\label{prop:autmorph_pres_gen_cart}
Let $\sB \subseteq \sA$ be a unital inclusion of (not necessarily finite-dimensional) $C^*$-algebras, and let $\sigma \in \textrm{Aut}(\sA)$. Then $\sB$ is generalised Cartan in $\sA$ if and only if $\sigma(\sB)$ is generalised Cartan in $\sA$.
\end{prop}
\begin{proof}
    This follows from Lemma \ref{lem:autmorph_pres_maxprime}, \cite[Proposition 3.9]{BGK2026}, and Lemma \ref{lem:faithful_exp}.
\end{proof}

The following theorem provides a complete characterization of generalized Cartan subalgebras in finite-dimensional $C^*$-algebras.
\begin{theorem}\label{thm:char_gen_cartan}
    Let $\sB \subseteq \sA$ be a unital inclusion of finite-dimensional $C^*$-algebras. Then the following statements are equivalent:
    \begin{itemize}
        \item[(i)] $\sB$ is a generalised Cartan subalgebra of $\sA$.
        \item[(ii)] $\sB$ is regular in $\sA$ and satisfies \maxprime.
        \item[(iii)] $\sB$ satisfies \maxprime.
        \item[(iv)] $\sB' \cap \sA \subseteq \sB$.
        \item[(v)] $\sB' \cap \sA = \sZ(\sB)$.
        \item[(vi)] $\sB$ is multiplicity free.
        \item[(vii)] there is a unique conditional expectation from $\sA$ onto $\sB$.
    \end{itemize}
\end{theorem}
\begin{proof}

\noindent($(ii) \implies (i)$). This follows from Lemma \ref{lem:faithful_exp}. Furthermore, $(i) \implies (ii)$ is trivial.

    It is immediate that $(iv) \iff (v)$. Moreover, the equivalence $(iv) \iff (vi)$ follows from Lemma~\ref{lem:mult_free_iff_rel_com}, while $(iv) \iff (vii)$ is established in Lemma~\ref{lem:rel_com_iff_uniq_exp}. Therefore, 
    \[
    (i) \iff (ii)  \text{ and } (iv) \iff (v) \iff (vi) \iff (vii).
    \]
    We now prove the equivalence $(iii) \iff (iv)$.\\

    \noindent($(iii) \implies (iv)$). Let $\sB$ satisfies \maxprime and $c \in \sB'\cap \sA$. Consider the trivial virtual commutant $(\sB, \varphi)$ where $\varphi(x) := cx$ for $x \in \sB$. Due to \maxprime we get $\varphi(\sB) \subseteq \sB$. Hence:
    \[
    c = c \cdot 1_\sB = \varphi(1_\sB) \in \sB.
    \]

    \vspace{0.3cm}
    \noindent($(iv) \implies (iii)$). Assume $\sB'\cap \sA \subseteq \sB$. Suppose $(\sJ, \varphi)$ is a virtual commutatnt of $\sB$ in $\sA$. Since $\sJ$ is an ideal of $\sB$ and $\sB$ is finite dimensional so there exists a unique projection $p\in \sZ (\sB)$ such that $\sJ = p\sB$.

    Let $c:=\varphi(p)\in \sA$. See that $c$ commutes with $\sB$. Indeed, for $b\in \sB$:
    \[
    cb = \varphi(p) b = \varphi(pb) = \varphi(bp) = bc.
    \]
    Therefore $c\in \sB' \cap \sA$. Thus by $(ii)$ we get $c\in \sB$. Now let $x \in \sJ$ be arbitrary. Then $x = px$. Hence,
    \[
    \varphi(x) = \varphi(px) = \varphi(p) x = cx \in \sB.
    \]
    Therefore, $\varphi(\sJ) \subseteq \sB$. Since $(\sJ, \varphi)$ was chosen arbitrary, $\sB$ satisfies \maxprime.

    Hence we proved:
    \[
     (i) \iff (ii)  \text{ and } (iii) \iff (iv) \iff (v) \iff (vi) \iff (vii).
    \]
    Now note that $(ii) \implies (iii)$ is obvious. Finally,

    \vspace{0.3cm}
    \noindent($(vi) \implies (ii)$). Indeed we have seen $(vi)\implies (iii)$, so $\sB$ satisfies \maxprime. Also every multiplicity free algebra is regular (by Lemma \ref{lem:mult-free-is-regular}). Thus $(vi) \implies (ii)$.

    This completes the proof!
\end{proof}

By Theorem \ref{thm:char_gen_cartan}, generalized Cartan subalgebras are precisely the multiplicity free subalgebras. Furthermore, Proposition \ref{prop:num_of_mult_free} provides an explicit count of the multiplicity free embeddings of $\mn{\vec{d}}$ into $\mn{\vec{n}}$, and hence determines the number of Cartan embeddings of $\mn{\vec{d}}$ into $\mn{\vec{n}}$. We conclude our study by listing, in Table~\ref{tab:matrix_counts}, the total number of Cartan embeddings of $\mn{\vec{d}}$ into $\mn{\vec{n}}$ for several choices of $\vec{d}$ and $\vec{n}$, as determined by Proposition~\ref{prop:num_of_mult_free}. The values in the table were computed using a short script.

\begin{table}[htbp]
\centering\label{tab:multi-free}
\begin{tabular}{lllc}
\hline
$\vec{d}$ & $\vec{n}$ & \textbf{$S_d = S_n$?} & \textbf{\makecell{\small Number of Cartan Embeddings}} \\
\hline
$[1, 2, 3, 4]$ & $[5, 5]$ & Yes ($10 = 10$) & 2 \\
$[2, 2, 2, 2, 2]$ & $[4, 6]$ & Yes ($10 = 10$) & 10 \\
$[1, 1, 2, 2]$ & $[3, 3]$ & Yes ($6 = 6$) & 4 \\
$[1, 1, 1, 3]$ & $[2, 1, 3]$ & Yes ($6 = 6$) & 3 \\
$[1, 2, 3, 4, 5]$ & $[7, 8]$ & Yes ($15 = 15$) & 3 \\
$[1, 2, 2, 3]$ & $[4, 4]$ & Yes ($8 = 8$) & 2 \\
$[1, 2, 4]$ & $[7]$ & Yes ($7 = 7$) & 1 \\
$[3, 4, 5]$ & $[6, 6]$ & Yes ($12 = 12$) & 0 \\
$[1, 5, 6]$ & $[4, 8]$ & Yes ($12 = 12$) & 0 \\
$[2, 2, 2]$ & $[4, 4]$ & No ($6 \neq 8$) & 0 \\
$[1, 1, 1, 1]$ & $[3, 2]$ & No ($4 \neq 5$) & 0 \\
$[1, 1, 1, 1, 1]$ & $[1, 1, 1, 1, 1]$ & Yes ($5 = 5$) & 120 \\
\hline
\end{tabular}

\vspace{0.5cm}
\caption{Total number of multiplicity-free matrices for various fixed $\vec{d}$ and $\vec{n}$.}
\label{tab:matrix_counts}
\end{table}

    \section{Acknowledgement}
    We would like to thank Prof. Keshab Chandra Bakshi for his valuable discussions and insightful suggestions for this project. Sumit Kumar would like to thank Prof. Keshab Chandra Bakshi for the opportunity to work at IIT Kanpur under his guidance through project no. SPO/ANRF/MATH/2025271. 

    \medskip
	
	\bibliographystyle{amsalpha} 
	\bibliography{references}

\end{document}